\documentclass{amsart}
\numberwithin{equation}{section}
\usepackage[utf8]{inputenc}
\usepackage[T1]{fontenc}
\usepackage{lmodern}
\usepackage[margin=1in]{geometry}
\usepackage[usenames]{xcolor}
\usepackage{mathtools}
\usepackage{amssymb}
\usepackage{amsthm}
\usepackage{xparse}
\usepackage{xspace}
\usepackage[all]{xy}
\usepackage[backref=page]{hyperref}
\usepackage[capitalize, noabbrev]{cleveref}

\DeclareDocumentCommand{\shortexact}{s O{} O{} mmmm}{
\IfBooleanTF{#1}{ 
 \xymatrix{
  1\ar[r] & #4\ar[r]^-{#2} & #5\ar[r]^-{#3} & #6\ar[r] & 1#7
 }
}{ 
 \xymatrix{
  0\ar[r] & #4\ar[r]^-{#2} & #5\ar[r]^-{#3} & #6\ar[r] & 0#7
 }
}}

\newcommand{\Z}{\mathbb Z}
\newcommand{\Q}{\mathbb Q}
\newcommand{\R}{\mathbb R}
\newcommand{\RP}{\mathbb{RP}}
\renewcommand{\O}{\mathrm O}
\newcommand{\SO}{\mathrm{SO}}
\newcommand{\Spin}{\mathrm{Spin}}
\newcommand{\Pin}{\mathrm{Pin}}

\newcommand{\Pinp}{\relax\ifmmode{\Pin^+}\else Pin\textsuperscript{$+$}\xspace\fi}
\newcommand{\pinp}{pin\textsuperscript{$+$}\xspace}
\newcommand{\Pinm}{\relax\ifmmode{\Pin^-}\else Pin\textsuperscript{$-$}\xspace\fi}
\newcommand{\pinm}{pin\textsuperscript{$-$}\xspace}

\newcommand{\csum}{\mathop{\#}}
\DeclareMathOperator{\Det}{Det}
\newcommand{\Aut}{\mathrm{Aut}}
\DeclareMathOperator{\Hom}{Hom}
\DeclareMathOperator{\Ext}{Ext}
\newcommand{\Top}{\mathrm{Top}}
\newcommand{\STop}{\mathrm{STop}}
\newcommand{\TopSpin}{\mathrm{TopSpin}}
\newcommand{\TopPin}{\mathrm{TopPin}}

\newcommand{\surj}{\twoheadrightarrow}

\newtheorem{thm}[equation]{Theorem}
\newtheorem*{thm*}{Theorem}
\newtheorem*{Thomthm}{\cref{Thom_simplification}}
\newtheorem{lem}[equation]{Lemma}

\newtheorem{prop}[equation]{Proposition}

\theoremstyle{definition}
\newtheorem{exm}[equation]{Example}
\newtheorem{defn}[equation]{Definition}

\theoremstyle{remark}
\newtheorem{rem}[equation]{Remark}

\crefname{thm}{Theorem}{Theorems}
\crefname{lem}{Lemma}{Lemmas}
\crefname{cor}{Corollary}{Corollaries}
\crefname{prop}{Proposition}{Propositions}
\crefname{ex}{Exercise}{Exercises}
\crefname{exm}{Example}{Examples}
\crefname{defn}{Definition}{Definitions}
\crefname{claim}{Claim}{Claims}
\crefname{rem}{Remark}{Remarks}
\crefname{fct}{Fact}{Facts}
\crefname{note}{Note}{Notes}

\newcommand{\term}{\emph} 

\DeclarePairedDelimiter\abs{\lvert}{\rvert}
\DeclarePairedDelimiter\set{\{}{\}}
\newcommand{\vp}{\varphi}
\newcommand{\inj}{\hookrightarrow}

\newcommand{\NewThomSpectrum}[1]{\expandafter\newcommand\csname M#1\endcsname{\mathit{M#1}}}
\newcommand{\NewMTSpectrum}[1]{\expandafter\newcommand\csname MT#1\endcsname{\mathit{MT#1}}}
\newcommand{\BothThomSpectra}[1]{\NewThomSpectrum{#1}\NewMTSpectrum{#1}}
\BothThomSpectra{O}
\BothThomSpectra{SO}
\BothThomSpectra{Spin}
\BothThomSpectra{Pin}
\BothThomSpectra{Top}
\BothThomSpectra{STop}
\BothThomSpectra{TopSpin}
\BothThomSpectra{TopPin}

\usepackage{microtype}
\usepackage{hyperref}
\hypersetup{
 colorlinks,
 linkcolor={red!50!black},
 citecolor={green!50!black},
 urlcolor={blue!80!black}
}

\title{Stable diffeomorphism classification of some unorientable 4-manifolds}
\author{Arun Debray}
\address{Department of Mathematics, University of Texas, Austin, Texas 78712}
\email{a.debray@math.utexas.edu}

\date{\today}

\begin{document}
\maketitle

\begin{abstract}
Kreck's modified surgery theory reduces the classification of closed, connected 4-manifolds, up to connect sum with
some number of copies of $S^2\times S^2$, to a series of bordism questions. We implement this in the case of
unorientable 4-manifolds $M$ and show that for some choices of fundamental groups, the computations simplify
considerably. We use this to solve some cases in which $\pi_1(M)$ is finite of order 2 mod 4: under an assumption on
cohomology, there are nine stable diffeomorphism classes for which $M$ is pin$^+$, one stable diffeomorphism class
for which $M$ is pin$^-$, and four stable diffeomorphism classes for which $M$ is neither. We also determine the
corresponding stable homeomorphism classes.
\end{abstract}


\setcounter{section}{-1}
\section{Introduction}
	The classification of closed $4$-manifolds up to diffeomorphism is impossible in general: a solution would also
solve the word problem for groups. Even if one fixes the fundamental group to avoid this problem, the
classification is still currently intractable. For this reason, topologists study weaker classifications of
$4$-manifolds which are coarse enough to be calculable yet fine enough to be useful.

Stable diffeomorphism is an example of such an invariant. Two closed $4$-manifolds $M$ and $N$ are \term{stably
diffeomorphic} if there are $m,n\ge 0$ such that $M\csum m(S^2\times S^2)$ is diffeomorphic to $N\csum n(S^2\times
S^2)$. This notion of equivalence has applications to quantum topology: for example, Reutter~\cite[Theorem
A]{Reu20} shows that the partition functions of 4d semisimple oriented TFTs are insensitive to stable
diffeomorphism along the way to showing that such TFTs cannot distinguish homotopy-equivalent closed, oriented
$4$-manifolds. And stable diffeomorphism classes are computable: once the fundamental group $G$ is fixed,
Kreck~\cite{Kre99} shows how to reduce the classification of $4$-manifolds up to stable diffeomorphism to a
collection of bordism computations, and for many choices of $G$, the classification of closed, connected, oriented
$4$-manifolds with $\pi_1(M)\cong G$ up to stable diffeomorphism has been completely worked out, thanks to work of
Wall~\cite{Wal64}, Teichner~\cite{Tei92}, Spaggiari~\cite{Spa03}, Crowley-Sixt~\cite{CS11},
Politarczyk~\cite{Pol13}, Kasprowski-Land-Powell-Teichner~\cite{KLPT17}, Pedrotti~\cite{Ped17},
Hambleton-Hildum~\cite{HH19}, and Kasprowski-Powell-Teichner~\cite{KPT20}.

Researchers interested in topological manifolds also study \term{stable homeomorphism} of topological manifolds,
i.e.\ homeomorphism after connect-summing with some number of copies of $S^2\times S^2$. Kreck's theorem applies to
this case too, reframing the question in terms of bordism of topological manifolds. Stable homeomorphism
classifications are studied by Teichner~\cite[\S 5]{Tei92}, Wang~\cite{Wan95},
Hambleton-Kreck-Teichner~\cite{HKT09}, Kasprowski-Land-Powell-Teichner~\cite[\S\S 4--5]{KLPT17},
Hambleton-Hildum~\cite{HH19}, and Kasprowski-Powell-Teichner~\cite[\S 2.3]{KPT20},

Much less work has been done on unorientable $4$-manifolds, even though the theory still works and is simpler in
some cases, as we explain below. There is some work in the literature, such as that of Kreck~\cite{Kre84},
Wang~\cite{Wan95}, Kurazono~\cite{Kur01}, Davis~\cite{Dav05}, and Friedl-Nagel-Orson-Powell~\cite[\S 12]{FNOP19}.

The goal of this paper is to compute sets of stable diffeomorphism and stable homeomorphism classes for a class of
unorientable $4$-manifolds, as well as determining the corresponding complete stable diffeomorphism and
homeomorphism invariants. As a consequence of our \cref{Thom_simplification}, for many finite groups $G$, the
classification of stable diffeomorphism or homeomorphism classes of unorientable $4$-manifolds with $\pi_1(M) \cong
G$ reduces to the stable classifications for a smaller $2$-group. For example, we show that the stable
diffeomorphism, resp.\ homeomorphism classification when $\pi_1(M)\cong\Z/2$ determines the stable diffeomorphism,
resp.\ homeomorphism classification for some groups $G$ of order $2\bmod 4$. We then compute these classifications
using Kreck's techniques.

Suppose $G$ is the fundamental group of an unorientable manifold. Then there is an extension
\begin{equation}
\label{unor_SES}
	\shortexact*{K}{G}{\Z/2},
\end{equation}
where $G\surj\Z/2$ is defined by classifying loops as orientation-preserving or orientation-reversing. Therefore
$\Z/2$ acts on $K$.
\begin{thm*}[Main theorem]
Let $G$ be a finite group of order $2\bmod 4$, and suppose that in~\eqref{unor_SES}, $\Z/2$ acts trivially on
$H^*(BK)$.
\begin{enumerate}
	\item There are fourteen equivalence classes of closed, connected, unorientable $4$-manifolds $M$ up to stable
	diffeomorphism: nine for which $M$ is \pinp, one for which $M$ is \pinm, and four for which $M$ is neither.
	\item There are twenty equivalence classes of closed, connected, unorientable topological $4$-manifolds $M$ up
	to stable homeomorphism: ten for which $M$ is \pinp, two for which $M$ is \pinm, and eight for which $M$ is
	neither.
\end{enumerate}
\end{thm*}
This is a combination of \cref{almost_spin_classes,non_spin_classes,top_almost_spin,top_no_spin}. In those theorems
we also determine complete stable diffeomorphism/homeomorphism invariants for these manifolds. The classification
for $M$ neither \pinp or \pinm can be extracted from work of Davis~\cite[Theorem 2.3]{Dav05}, but the other parts
are new.

We prove these theorems by establishing isomorphisms of bordism groups. Specifically, Kreck's modified surgery
theory associates to $G$ a set of symmetry types $\xi\colon B\to B\O$ and expresses the set of stable
diffeomorphism classes in terms of the bordism groups $\Omega_4^\xi$; we show that when $\abs G \equiv 2\bmod 4$
and the assumption about $H^*(K)$ holds, the Thom spectra of these symmetry types are homotopy equivalent to the
Thom spectra for unoriented, \pinp, and \pinm bordism. In the smooth case, the bordism groups $\Omega_4^\O$,
$\Omega_4^{\Pin^+}$, and $\Omega_4^{\Pin^-}$ are well-known. The topological versions of these bordism groups are
less well-known, but Kirby-Taylor~\cite[\S 9]{KT90} compute $\Omega_4^{\TopPin^\pm}$ and provide enough information
for us to compute $\Omega_4^{\Top}$, which we do in \cref{top_4_bordism}.

The argument we use to establish the isomorphism from $\xi$-bordism to a simpler kind of bordism applies to more
general choices of $\pi_1(M)$.
\begin{Thomthm}
Suppose $G$ is a finite group fitting into an extension
\begin{equation}
	\shortexact*[][\vp]{K}{G}{P},
\end{equation}
where $\abs K$ is odd and $P$ is a $2$-group, and suppose $P$ acts trivially on $H^*(BK)$. For any unorientable
virtual vector bundle $V\to BP$, $\vp$ induces an equivalence of Thom spectra $(BG)^{\vp^*V}\overset\simeq\to
(BP)^V$.
\end{Thomthm}
The Pontrjagin-Thom construction turns this equivalence into isomorphisms of bordism groups from the unorientable
symmetry types Kreck associates to $G$ to the unorientable symmetry types for $P$, which we can use to compute
stable diffeomorphism classes. The proof strongly requires the assumption that $V$ is unorientable; nothing like
this is true in the oriented case.

Our main theorem above covers the case $\abs G\equiv 2\bmod 4$. The next step would be to consider
$P\cong\Z/2\times\Z/2$ or $\Z/4$, which would suffice for many groups $G$ of order $4\bmod 8$. For these choices of
$P$, many of the needed bordism groups have already been computed in the literature for other applications. For
$P\cong\Z/4$, see Botvinnik-Gilkey~\cite[\S 5]{BG97}; for $P\cong\Z/2\times\Z/2$, see work of
Guo-Ohmori-Putrov-Wan-Wang~\cite[\S 7]{GOPWW18}, the author in~\cite[Appendix F]{KPMTD20} and~\cite[\S 4.4]{Deb21},
and Wan-Wang-Zheng~\cite[Appendix A]{WWZ20}.

We begin in \S\ref{s_BG} with a quick review of Kreck's theorem~\cite{Kre99} on stable diffeomorphism classes of
$4$-manifolds within a given $1$-type. In \S\ref{s_simplify}, we study the Thom spectra of unorientable vector
bundles over $BG$, where $G$ is a finite group, proving \cref{Thom_simplification}. In \S\ref{s_2mod4}, we
specialize to the case where $\abs G \equiv 2\bmod 4$, determining the three possible normal $1$-types and
computing the sets of stable diffeomorphism classes for them. We prove \cref{almost_spin_classes,non_spin_classes},
which together form the smooth part of the main theorem above. In \cref{fake_RP4}, we discuss an example: $\RP^4$
is homeomorphic but not stably diffeomorphic to Cappell-Shaneson's fake $\RP^4$. This fact was known to
Cappell-Shaneson~\cite{CS71, CS76} and the proof using Kreck's surgery theory is due to Stolz~\cite{Sto88}. In
\S\ref{s_topl}, we consider stable homeomorphism classes of topological manifolds with $\abs{\pi_1(M)}\equiv 2\bmod
4$, and prove \cref{top_almost_spin,top_no_spin}, which form the topological part of the main theorem above.
\subsection*{Acknowledgments}
I thank my advisor, Dan Freed, for his constant help and guidance. I also would like to thank Matthias Kreck,
Riccardo Pedrotti, and Oscar Randal-Williams for some helpful conversations related to this paper.

A portion of this work was supported by the National Science Foundation under Grant No.\ 1440140 while the author
was in residence at the Mathematical Sciences Research Institute in Berkeley, California, during January--March
2020.


\section{Review: normal $1$-types, normal $1$-smoothings, and stable diffeomorphism classes}
	\label{s_BG}
We review some standard definitions in this area. We will always assume our manifolds are closed and connected.
Except in \S\ref{s_topl}, we also assume they are smooth.
\begin{defn}
A \term{normal $1$-type} of a manifold $M$ is a fibration $\xi\colon B\to B\O$ such that there is a lift of the map
$\nu\colon M\to BO$ classifying the stable normal bundle of $M$ to a map $\widetilde\nu\colon M\to B$ such that
$\xi\circ\widetilde\nu = \nu$, $\widetilde\nu$ is $2$-connected, and $\xi$ is $2$-coconnected.

A choice of such a lift is called a \term{normal $1$-smoothing} of $M$.
\end{defn}
Any two normal $1$-types of a given manifold are homotopy equivalent as spaces over $B\O$, so we will abuse
notation and say ``the'' normal $1$-type.

The map $\xi\colon B\to B\O$ determines a bordism theory of manifolds with a lift of the stable normal bundle
across $\xi$, which we denote $\Omega_*^\xi$; a normal $1$-smoothing of $M$ determines a class in this bordism
group. Different normal smoothings of the same manifold do not always define the same class in $\Omega_*^\xi$.

Let $V_\SO\to B\SO$, $V_\Spin\to B\Spin$, etc., denote the tautological stable vector bundles over their respective
spaces. We use the convention that maps to $B\O$ are represented by rank-zero virtual vector bundles, which is why
we write $E - \dim E$ in~\eqref{alspineq}, for example.
\begin{exm}[{Kreck~\cite[\S 2, Proposition 2]{Kre99}}]\let\qed\relax
\label{krecks_examples}
When $M$ is unorientable, Kreck classifies the possible normal $1$-types of $M$ into two families. Let $M'\to M$ be
the universal cover of $M$, which is classified by a map $\theta\colon M\to B\pi_1(M)$.
\begin{description}
	\item[Almost spin] If $M'$ admits a spin structure, $M$ is called \term{almost spin}. In this case, $w_1(M) =
	\theta^*x_1$ and $w_2(M) = \theta^*x_2$ for some $x_1,x_2\in H^*(BG;\Z/2)$. Assume there is a vector bundle
	$E\to BG$ such that $w_i(E) = x_i$ for $i = 1,2$.\footnote{This will be true for all cases we consider in this
	paper, but is not true in general.} Then, the normal $1$-type of $M$ is
	\begin{equation}
	\label{alspineq}
	\begin{gathered}
		\xymatrix{
			& B\Spin\times B\pi_1(M)\ar[d]^{V_\Spin\oplus (E-\dim E)}\\
			M\ar[r]_-\nu\ar[ur] & B\O.
		}
	\end{gathered}
	\end{equation}
	\item[Totally non-spin] If $M'$ does not admit a spin structure, $M$ is called \term{totally non-spin}. In this
	case, $w_1(M) = \theta^*x$ for some $x\in H^1(BG;\Z/2)$. Let $E\to BG$ be a line bundle with $w_1(E) = x$. Then
	the normal $1$-type of $M$ is
	\begin{equation}
	\begin{gathered}
		\xymatrix{
			& B\SO\times B\pi_1(M)\ar[d]^{V_\SO\oplus (E-1)}\\
			M\ar[r]_-\nu\ar[ur] & B\O.
		}
	\end{gathered}
	\end{equation}
\end{description}
\end{exm}
Because $S^2\times S^2$ has trivial stable normal bundle, taking connect sum with $S^2\times S^2$ does not change
the normal $1$-type of a $4$-manifold; thus the classification of $4$-manifolds up to stable diffeomorphism
can proceed one normal $1$-type at a time. Moreover, because $S^2\times S^2$ is null-bordant, one might
conclude that stably diffeomorphic $4$-manifolds $M$ and $N$ are bordant --- or, more precisely, that $M$ and $N$
admit normal $1$-smoothings which are bordant in $\Omega_4^\xi$. So a plausible lower bound for the set of stable
diffeomorphism classes with normal $1$-type $\xi$ would be $\Omega_4^\xi$ modulo some identifications arising from
inequivalent normal $1$-smoothings of the same underlying manifold. Remarkably, this turns out to be a complete
classification!
\begin{thm}[{Kreck~\cite[Theorem C; \S 3, Proposition 4]{Kre99}}]\hfill
\label{its_bordism}
\begin{enumerate}
	\item If $M$ and $N$ are $4$-manifolds of the same normal $1$-type $\xi$ admitting normal $1$-smoothings
	which are bordant in $\Omega_4^\xi$, then $M$ is stably diffeomorphic to $N$.
	\item If $\pi_1(\xi)$ is finite, every class in $\Omega_4^\xi$ can be realized as the normal $1$-smoothing of
	a $4$-manifold with normal $1$-type $\xi$.
\end{enumerate}
The upshot is that if $\Aut(\xi)$ denotes the group of fiber homotopy equivalences of $\xi\to B\O$, the set of
stable diffeomorphism classes of $4$-manifolds with normal $1$-type $\xi$ is $\Omega_4^\xi/\Aut(\xi)$.
\end{thm}
The set of bordism classes of normal $1$-smoothings of a given $4$-manifold is contained within an
$\Aut(\xi)$-orbit of $\Omega_4^\xi$, so one effect of the quotient is to identify these as all coming from the same
manifold.

This illustrates the standard way to calculate stable diffeomorphism classes: determine $\Omega_4^\xi$, then
determine the $\Aut(\xi)$-action. These bordism groups are the homotopy groups of the Thom spectrum $M\xi$ of
$\xi$, so in the next section we begin the calculation of stable diffeomorphism classes by simplifying $M\xi$.

\section{Simplifying Thom spectra}
	\label{s_simplify}
\Cref{its_bordism} tells us to investigate the Thom spectra of the normal $1$-types in \cref{krecks_examples}. In
both cases, the vector bundle is an exterior direct sum, so the Thom spectra split, as $\MSpin\wedge (B\pi_1(M))^V$
in the almost spin case and $\MSO\wedge (BG)^V$ in the totally non-spin case, where $V$ is a rank-zero unoriented
virtual vector bundle. We attack the problem by simplifying $(B\pi_1(M))^V$ for some choices of $\pi_1(M)$.

\begin{thm}
\label{Thom_simplification}
Suppose $G$ is a finite group fitting into an extension
\begin{equation}
\label{G_as_SES}
	\shortexact*[][\vp]{K}{G}{P},
\end{equation}
where $\abs K$ is odd and $P$ is a $2$-group, and suppose $P$ acts trivially on $H^*(BK)$. For any unorientable
virtual vector bundle $V\to BP$, $\vp$ induces an equivalence of Thom spectra $(BG)^{\vp^*V}\overset\simeq\to
(BP)^V$.
\end{thm}
We'll prove this in a series of lemmas.
\begin{defn}
\label{locsysdef}
Let $H$ be a group, $A$ be an abelian group, and $\alpha\in H^1(BH;\Z/2)$. Using the identification
$H^1(BH;\Z/2)\cong\Hom(H, \Z/2)$, let $A_\alpha$ be the $\Z[H]$-module which is the abelian group $\Z$ with the
$H$-action in which $g\in H$ acts by $(-1)^{\alpha(g)}$.
\end{defn}
\begin{lem}
\label{locsyscoh}
In the situation of \cref{Thom_simplification}, both $\widetilde H^*((BG)^{\vp^*V})$ and $\widetilde H^*((BP)^V)$
are $2$-torsion.
\end{lem}
\begin{proof}
Using \cref{locsysdef}, we define the $\Z[P]$-module $A_{w_1(V)}$ and the $\Z[G]$-module $A_{w_1(\vp^*P)}$, which
is isomorphic to the pullback of $A_{w_1(V)}$ by $\vp$. The Thom isomorphism provides isomorphisms of graded
abelian groups
\begin{subequations}
\begin{align}
	H^*(BP; \Z_{w_1(V)}) &\overset\cong\longrightarrow\widetilde H^*((BP)^V)\\
	H^*(BG;\Z_{w_1(\vp^*V)}) &\overset\cong\longrightarrow\widetilde H^*((BG)^{\vp^*V}),
\end{align}
\end{subequations}
so we will prove the lemma using group
cohomology -- specifically, using the Lyndon-Hochschild-Serre spectral sequence
\begin{equation}
\label{the1stLHS}
	E_2^{p,q} = H^p(BP; (H^q(BK;\Z))_{w_1(V)}) \Longrightarrow H^{p+q}(BG; \Z_{w_1(\vp^*V)}).
\end{equation}
Here it is crucial that $P$ acts trivially on $H^*(BK)$; otherwise we would have a different local coefficient
system than $H^q(BK;\Z)_{w_1(V)}$ in~\eqref{the1stLHS}.

Since $E_2^{p,q}\cong H^p(BP; M_q)$ for some $\Z[P]$-module $M_q$, $E_2^{p,q}$ is $2$-torsion for $p > 1$ by
Maschke's theorem.\footnote{We use Maschke's theorem as follows: if $G$ is a finite group and $k$ is a field of
characteristic $0$ or characteristic $\ell\nmid \#G$, the category of $k[G]$-modules is semisimple. Therefore all
positive-degree Ext groups vanish, in particular $H^m(BG; M)\cong\Ext_{k[G]}^m(\Z, M)$ for any $k[G]$-module $M$
and $m > 1$. Combined with the universal coefficient theorem, this implies that for any $\Z[G]$-module $M$ and $m >
1$, $H^m(G; M)$ is torsion ($k = \Q$), and lacks $\ell$-torsion if $\ell\nmid \#G$.} When $p = 0$,
\begin{equation}
	E_2^{0,q}\cong H^0(BP; H^q(BK)_{w_1(V)})\cong (H^q(BK)_{w_1(V)})^P.
\end{equation}
We will show this vanishes. First, $H^q(BK)$ is $\Z$ for $q = 0$ and is odd-primary torsion for $q > 0$ (by
Maschke's theorem, because $2\nmid\#K$). Therefore if $a\in H^q(BK)$ and $-a = a$, $a = 0$.  Since $w_1(V)\ne 0$,
there is some $g\in P$ which acts on $\Z_{w_1(V)}$ as $-1$, hence also acts on $H^q(BK)_{w_1(V)}$ as $-1$, so the
subgroup of invariants of $H^q(BK)_{w_1(V)}$ is $\set 0$.

Considering the line $q = 0$ proves $H^*(BP;\Z_{w_1(V)})$ is $2$-torsion. For $H^*(BG;\Z_{w_1(\vp^*V)})$, we have
shown the $E_2$-page is $2$-torsion, so the graded abelian group the spectral sequence converges to is also
$2$-torsion.
\end{proof}
\begin{lem}
\label{cohring}
With $G$ and $P$ as in \cref{Thom_simplification}, $\vp^*\colon H^*(BP;\Z/2)\to H^*(BG;\Z/2)$ is an isomorphism of
graded rings.
\end{lem}
\begin{proof}
Since $K$ has odd order, its mod $2$ cohomology is $\Z/2$ in degree $0$ and vanishes elsewhere, so the result
follows from the Leray-Hirsch theorem applied to the fibration $BK \to BG\to BP$ induced by~\eqref{G_as_SES}.
\end{proof}
\begin{proof}[Proof of \cref{Thom_simplification}]
Use the homology Whitehead theorem: if $f\colon X\to Y$ is a map of bounded-below spectra which induces an
isomorphism on rational cohomology and on mod $p$ cohomology for every prime $p$, then $f$ is a homotopy
equivalence. \Cref{locsyscoh} and the universal coefficient theorem imply that if $k = \Q$ or $k = \Z/p$ for an odd
prime $p$, $\widetilde H^*((BG)^{\vp^*V};k)$ and $\widetilde H^*((BP)^V; k)$ both vanish, so the map between them
is vacuously an isomorphism. The sole remaining case is $p = 2$. Since $1\equiv -1\bmod 2$, $(\Z/2)_{w_1(V)}$
carries the trivial $P$-action; thus, the Thom isomorphism has the form
\begin{subequations}
\begin{equation}
	H^*(BP;\Z/2)\overset\cong\longrightarrow \widetilde H^*((BP)^V;\Z/2).
\end{equation}
Analogously, there is a Thom isomorphism
\begin{equation}
	H^*(BG;\Z/2))\overset\cong\longrightarrow \widetilde H^*((BG)^{\vp^*V};\Z/2).
\end{equation}
\end{subequations}
As the Thom isomorphism is functorial with respect to pullbacks of vector bundles, \cref{cohring} lifts to imply
that
\begin{equation}
	\vp^*\colon \widetilde H^*((BP)^V;\Z/2)\longrightarrow \widetilde H^*((BG)^{\vp^*V};\Z/2)
\end{equation}
is an isomorphism.
\end{proof}

\section{The case $\abs{\pi_1(X)}\equiv 2\bmod 4$}
	\label{s_2mod4}
If $M$ is an unorientable manifold, the description of loops as orientation-preserving or orientation-reversing
defines a surjection $p\colon \pi_1(M)\to\Z/2$, so $\pi_1(M)$ cannot have odd order. Thus the simplest case occurs
when $\abs{\pi_1(M)}\equiv 2\bmod 4$, so that $\abs{\ker(p)}$ is odd. For the rest of this section, fix such a
group $G$, and assume that $\Z/2$ acts trivially on $H^*(B\ker(p))$.

In this case, \cref{Thom_simplification} applies to show that if $\Z/2$ acts trivially on $H^*(B\ker(p))$ and $V\to
B\Z/2$ is any unorientable virtual vector bundle, the map $(B\pi_1(M))^{p^*V}\overset\simeq\to (B\Z/2)^V$ is an
equivalence.

Let $\sigma\to B\Z/2$ denote the tautological line bundle and $x\coloneqq w_1(\sigma)\in H^1(B\Z/2;\Z/2)$, so
$H^*(B\Z/2; \Z/2) \cong \Z/2[x]$. Because $\ker(p)$ has odd order, the Leray-Hirsch theorem implies $p^*\colon
H^*(B\Z/2;\Z/2)\to H^*(B\pi_1(M);\Z/2)$ is an isomorphism.
\subsection{The almost spin case}
\Cref{krecks_examples} shows there are two unorientable normal $1$-types in this case: $w_1(\nu)\ne 0$, so
it must be the pullback of $p^*x\in H^1(B\pi_1(M);\Z/2)$, and for $w_2$, we have two choices: $w_2 = 0$ (the normal
bundle is \pinp) and $w_2 = p^*x^2$ (the normal bundle is \pinm).

Recall that for a manifold $M$, $M$ is pin\textsuperscript{$\pm$} (i.e.\ the tangent bundle is
pin\textsuperscript{$\pm$}) iff the normal bundle is pin\textsuperscript{$\mp$}. A (tangential) \pinp $4$-manifold
$M$ has a $\Z/16$-valued invariant given by the $\eta$-invariant of a twisted Dirac operator~\cite[\S 4]{Sto88};
let $\eta'$ be the invariant assigning to a \pinp $4$-manifold $M$ the image of this $\eta$-invariant in the
nine-element set $(\Z/16)/(x\sim -x)$. We will see in the proof of \cref{almost_spin_classes} that all \pinp
structures on $M$ give the same value of $\eta'$, so we may define it as an invariant of manifolds which admit a
\pinp structure, without choosing such a structure.
\begin{thm}
\label{almost_spin_classes}
There are nine stable diffeomorphism classes of unorientable $4$-manifolds with $\pi_1(M)\cong G$ that admit a
(tangential) \pinp structure, and there is a single stable diffeomorphism class of manifolds with $\pi_1(M)\cong G$
that admit a (tangential) \pinm structure. In the \pinp case, $\eta'$ is a complete stable diffeomorphism
invariant.
\end{thm}
\begin{proof}
Both choices of $(w_1, w_2)$ arise from vector bundles: $(p^*x, 0)$ from $p^*\sigma$, and $(p^*x, p^*x^2)$ from
$p^*(3\sigma)$. Thus the normal $1$-types are
\begin{subequations}
\begin{align}
	V_{\Spin}\oplus (p^*\sigma - 1)\colon B\Spin\times B\pi_1(M) &\longrightarrow B\O\\
	V_{\Spin}\oplus (p^*(3\sigma) - 3)\colon B\Spin\times B\pi_1(M) &\longrightarrow B\O,
\end{align}
\end{subequations}
and their Thom spectra are $\MSpin\wedge (B\pi_1(M))^{p^*\sigma-1}$, resp.\ $\MSpin\wedge
(B\pi_1(M))^{p^*(3\sigma)-3}$. By \cref{Thom_simplification}, these are equivalent to $\MSpin\wedge
(B\Z/2)^{\sigma-1}$, resp.\ $\MSpin\wedge (B\Z/2)^{3\sigma-3}$.
\begin{thm}[{Peterson~\cite[\S 7]{Pet68}, Kirby-Taylor~\cite[Lemma 6]{KT90Pinp}}]
\label{smooth_pin_shear}
There are equivalences $\MSpin\wedge (B\Z/2)^{\sigma-1}\simeq\MTPin^-$ and $\MSpin\wedge (B\Z/2)^{3\sigma-3}\simeq
\MTPin^+$.\footnote{There is an important subtlety in the names of these spectra in the literature: $\MPin^\pm$
denotes the Thom spectra classifying pin\textsuperscript{$\pm$} structures on the stable normal bundle, and
$\MTPin^\pm$ denotes the Thom spectra classifying pin\textsuperscript{$\pm$} structures on the stable tangent
bundle. There are equivalences $\MPin^\pm\simeq\MTPin^\mp$. Information on pin\textsuperscript{$\pm$} bordism is
usually stated in terms of $\MTPin^\pm$.}
\end{thm}
These bordism groups are known.
\begin{itemize}
	\item In the case $w_2(\nu) = 0$, $\Omega_4^\xi\cong\Omega_4^{\Pin^-}\cong 0$~\cite{ABP69, KT90} --- all
	$4$-manifolds with this normal $1$-type are stably diffeomorphic.
	\item When $w_2(\nu) = p^*x^2$, $\Omega_4^\xi\cong\Omega_4^{\Pin^+}\cong\Z/16$~\cite{Gia73a, KT90Pinp, KT90}.
\end{itemize}
In the latter case we have to determine the $\Aut(\xi)$-action. $\RP^4$ admits two \pinp structures, and
Kirby-Taylor~\cite[Theorem 5.2]{KT90} choose an isomorphism $\Omega_4^{\Pin^+}\overset\cong\to\Z/16$ sending these
two \pinp structures to $\pm 1$. Therefore for any equivalence class $x\in\Omega_4^{\Pin^+}$ and any
$g\in\Aut(\Pin^+)$, $g\cdot x = \pm x$, because we can represent $x$ as a disjoint union of copies of $\RP^4$ with
some \pinp structure, and the $\Aut(\Pin^+)$-orbit of the $\RP^4$s is $\set{\pm 1}$. The isomorphism from
$\xi$-bordism to \pinp bordism allows us to also deduce that the $\Aut(\xi)$-orbit of a class $[M]$ in
$\Omega_4^\xi$ is $\set{\pm[M]}$. We obtain nine equivalence classes: $0, \pm 1, \pm 2,\dotsc,\pm 7, 8$, detected
by the image of the $\eta$-invariant in $(\Z/16)/(x\sim -x)$.
\end{proof}
As a consequence of Kreck's classification in \cref{krecks_examples}, we have seen that all unorientable, almost
spin $4$-manifolds $M$ with $\pi_1(M)\cong G$ are either \pinp or \pinm, and that this determines their normal
$1$-type. This is not true for more general $G$.
\begin{exm}
\label{fake_RP4}
Cappell-Shaneson~\cite{CS71, CS76} construct a closed, smooth $4$-manifold $Q$ that is homeomorphic but not
diffeomorphic to $\RP^4$, and show that $Q$ and $\RP^4$ are not stably diffeomorphic. Stolz~\cite{Sto88} gives
another proof of this fact by computing the classes of $\RP^4$ and $Q$ in $\Omega_4^\xi/\Aut(\xi)$. We briefly
summarize Stolz' proof.

Since $\pi_1(\RP^4)\cong\Z/2$ and $w_2(\RP^4) = 0$, the proof of
\cref{almost_spin_classes} shows $M\xi\simeq\MTPin^+$, $\Omega_4^\xi\cong\Z/16$, and the set of stable
diffeomorphism classes is $\Omega_4^\xi/\Aut(\xi)\cong (\Z/16)/(x\sim -x)$. Stolz~\cite{Sto88} chooses an
isomorphism $\Omega_4^\xi\overset\cong\to\Z/16$ and shows that it sends the two \pinp structures on $\RP^4$
to $\pm 1$ and the two \pinp structures on $Q$ to $\pm 9$; therefore $\RP^4$ and $Q$ are not stably
diffeomorphic.
\end{exm}
\subsection{The totally non-spin case}
\begin{thm}
\label{non_spin_classes}
There are four stable diffeomorphism classes of unorientable, totally non-spin $4$-manifolds with $\pi_1(M)\cong
G$. The Stiefel-Whitney numbers $w_4$ and $w_2^2$ detect these classes.
\end{thm}
This theorem can also be extracted from work of Davis~\cite[Theorem 2.3]{Dav05}, who computes a different set of
invariants.
\begin{proof}
\Cref{krecks_examples} shows there is only one unorientable normal $1$-type in this case: $w_1(\nu)\ne 0$, so it
must be pulled back from $p^*x\in H^1(B\pi_1(M);\Z/2)$. Since $p^*x = w_1(p^*\sigma)$, the normal $1$-type is
\begin{equation}
	V_{\SO}\oplus (p^*\sigma - 1)\colon B\SO\times B\pi_1(M)\longrightarrow B\O
\end{equation}
and its Thom spectrum is $\MSO\wedge (B\pi_1(M))^{p^*\sigma - 1}$, which by \cref{Thom_simplification} is
equivalent to $\MSO\wedge (B\Z/2)^{\sigma-1}$.
\begin{lem}[{Gray~\cite[\S 2]{Gra80}}]
\label{MSO_MO_splitting}
There is an equivalence $\MSO\wedge (B\Z/2)^{\sigma-1}\simeq\MO$.
\end{lem}
So $\Omega_4^\xi\cong\Omega_4^\O$, and $\Omega_4^\O\cong\Z/2\oplus\Z/2$~\cite[Corollaire following Théorème
IV.12]{ThomThesis}. The $\Aut(\xi)$-action is trivial. To see this, first observe that $\Aut(\mathrm{id}\colon
B\O\to B\O)$ is trivial, hence acts trivially on $\Omega_4^\O$. Thus the $\Aut(\xi)$-orbit of a class in
$\Omega_4^\xi$ maps to a single class in $\Omega_4^\O$, so $\Aut(\xi)$-orbits are singletons. Therefore any
complete bordism invariant for $\Omega_4^\O$ also is a complete stable diffeomorphism invariant for the normal
$1$-type $\xi$, such as $(w_2^2, w_4)$.
\end{proof}
\begin{rem}
\label{trichotomy}
If $M$ is \pinp or \pinm, then its double cover is spin, and hence $M$ is almost spin. So totally non-spin
manifolds are neither \pinp nor \pinm. Therefore the three normal $1$-types that occur when $\pi_1(M)\cong G$ and
$M$ is unorientable are the cases \pinp, \pinm, and neither \pinp nor \pinm.
\end{rem}


\section{Stable homeomorphism classes}
	\label{s_topl}
In order to classify stable homeomorphism classes of topological $4$-manifolds, we run the same story, replacing
$B\O$ with $B\Top$, where $\Top_n$ is the topological group of homeomorphisms $\R^n\to\R^n$ that fix the origin and
$\Top\coloneqq\varinjlim_n\Top_n$. As in the previous section, fix a group $G$ finite of order $2\bmod 4$ with a
surjective map $p\colon G\to\Z/2$, and assume that $\Z/2$ acts trivially on $H^*(B\ker(p))$.

Given a topological manifold $M$, there is a map $\nu\colon M\to B\Top$ called
the \term{stable topological normal bundle}, so we can define normal $1$-types, and Kreck's classification argument
still applies in the topological setting, this time determining stable homeomorphism classes.
\begin{lem}
\label{topological_examples}
Let $M$ be a closed, unorientable $4$-manifold. The possible normal $1$-types of $M$ are the same as in
\cref{krecks_examples}, except replacing $B\O$ with $B\Top$, $B\SO$ with $B\STop$, and $B\Spin$ with $B\TopSpin$.
\end{lem}
\begin{proof}
The proof is very similar to Kasprowski-Land-Powell-Teichner's determination of the possible normal $1$-types of
topological $4$-manifolds in the orientable case~\cite[Proposition 4.1]{KLPT17}. Since the Stiefel-Whitney classes
of a manifold are homotopy invariants, notions of almost spin and totally non-spin make sense for topological
manifolds. In the almost-spin case, we have to check that a lift $M\to B\TopSpin\times B\pi_1(M)$ is $2$-connected:
the proof is the same as in the smooth case, because $\pi_2(B\TopSpin) = 0$. For the totally non-spin case,
$\pi_1(B\STop) \cong\Z/2$, detected by $w_2$, and since $M$ is totally non-spin, $w_2(M)\ne 0$, so the lift is
surjective on $\pi_2$ just as in the smooth case.
\end{proof}
Our arguments below make use of the fact that bordism groups of topological manifolds are homotopy groups of Thom
spectra, which requires a transversality argument. In dimension $4$, Scharlemann~\cite{Sch76} proves the
topological transversality theorem that we need. See Teichner~\cite[\S IV]{Tei93} for more information.

Let $E_8$ denote Freedman's $E_8$ manifold~\cite{Fre82}. The obstruction to admitting a triangulation defines a
bordism invariant $\Omega_4^{\Top}\to\Z/2$~\cite[\S 9]{KT90} which is nonzero on $E_8$.
\subsection{The almost spin case}
There are topological versions of spin and pin\textsuperscript{$\pm$} structures; see Kirby-Taylor~\cite[\S
9]{KT90} for details. Kirby-Taylor also produce a homomorphism $S\colon
\Omega_4^{\TopPin^+}\to\Omega_2^{\TopPin^-}\cong\Omega_2^{\Pin^-}\cong\Z/8$ sending a \pinp topological
$4$-manifold $M$ to the \pinm bordism class of a continuously embedded representative of the Poincaré dual of
$w_1(M)^2$, which has an induced \pinm structure and a unique smooth structure. Let $S'$ be the invariant sending a
topological \pinp $4$-manifold $M$ to the image of $S(M)$ in the set $(\Z/8)/(x\sim -x)$.
\begin{thm}
\label{top_almost_spin}
\hfill
\begin{enumerate}
	\item There are ten stable homeomorphism classes of unorientable \pinp topological $4$-manifolds with
	$\pi_1(M)\cong G$. These classes are detected by the invariant $S'$ constructed above and the triangulation
	obstruction.
	\item There are two stable homeomorphism classes of unorientable \pinm topological $4$-manifolds with
	$\pi_1(M)\cong G$. These classes are detected by the triangulation obstruction.
\end{enumerate}
\end{thm}
\begin{proof}
Following the same line of argument as in the proof of \cref{almost_spin_classes}, the two normal $1$-types' Thom
spectra are $\MTopSpin\wedge (B\pi_1(M))^{p^*\sigma-1}$ and $\MTopSpin\wedge (B\pi_1(M))^{p^*(3\sigma)-3}$, and
\cref{Thom_simplification} simplifies these to $\MTopSpin\wedge (B\Z/2)^{\sigma-1}$ and $\MTopSpin\wedge
(B\Z/2)^{3\sigma-3}$, respectively.
\begin{lem}
There are equivalences $\MTopSpin\wedge (B\Z/2)^{\sigma-1}\simeq\MTTopPin^-$ and $\MTopSpin\wedge
(B\Z/2)^{3\sigma-3}\simeq\MTTopPin^+$.
\end{lem}
\begin{proof}
There are surjective maps $d_n\colon \Top_n\surj\set{\pm 1}$ given by assigning to a homeomorphism the automorphism
it defines on $H_n(\R_n, \R_n\setminus 0)\cong\Z$. These commute with the inclusions $\Top_n\inj\Top_{n+1}$, and
passing to the colimit defines a map $d\colon\Top\surj\set{\pm 1}$. This is a topological version of assigning an
orthogonal matrix its determinant, classifying whether it preserves or reverses orientation. Given a principal
$\Top$-bundle $P\to M$, let $\Det(P)\to M$ be the line bundle $P\times_{\Top}\R\to M$, where $\Top$ acts on $\R$
through $d$. The maps $\Top_n\times\O_1\to\Top_n\times\Top_1\to\Top_{n+1}$ allow us to make sense of ``$P\oplus
n\Det(P)$'' as a principal $\Top$-bundle.

We abuse notation for a moment to say that a $G$-structure on a principal $\Top$-bundle $P\to M$ is a reduction of
structure group of $P$ from $\Top$ to $G$. Then, just as in the smooth case, there is a natural equivalence between
the set of $\TopPin^-$-structures on $P$ and the set of $\TopSpin$ structures on $P\oplus\Det(P)$, and similarly
between the set of $\TopPin^+$-structures on $P$ and the set of $\TopSpin$ structures on $P\oplus 3\Det(P)$. The
proof is the same as in the smooth case. These equivalences are the only facts we need to know
about $\Pin^\pm$ in order to prove \cref{smooth_pin_shear} in the smooth setting, so the argument in the
topological setting can proceed in the same way.
%
%
\end{proof}
Therefore by \cref{Thom_simplification}, our two normal $1$-types are equivalent to $\MTTopPin^\pm$. The caveat
about switching between \pinp and \pinm when one passes between the tangent and normal bundles still applies here.
\begin{thm}[{Kirby-Taylor~\cite[Theorem 9.2]{KT90}}]\hfill
\begin{enumerate}
	\item $\Omega_4^{\TopPin^-}\cong\Z/2$, generated by $E_8$.
	\item $\Omega_4^{\TopPin^+}\cong \Z/8\oplus\Z/2$, with $\RP^4$ generating the $\Z/8$ summand and $E_8$ generating
	the $\Z/2$ summand.
	\item The map $\Omega_4^{\Pin^+}\to\Omega_4^{\TopPin^+}$ is identified with a map $\Z/16\to\Z/8\oplus\Z/2$
	which surjects onto the first factor and does not hit $E_8$.
	\item The homomorphism $S\colon \Omega_4^{\TopPin^+}\to\Omega_2^{\TopPin^-}\cong\Omega_2^{\Pin-}\cong\Z/8$
	sends $\RP^4$ to a generator.
\end{enumerate}
\end{thm}
Since $\Z/2$ is rigid, we conclude there are two stable homeomorphism classes in the \pinm case, detected by the
triangulation obstruction. For the \pinp case, the same line of reasoning in the proof of
\cref{almost_spin_classes} allows us to reduce to the case when $\xi$ is a topological \pinp structure, so we can
compute the action of $\Aut(\xi)$ on the generators. Since $E_8$ is
simply connected, it admits a unique topological \pinp structure, so is fixed by $\Aut(\xi)$. For $\RP^4$, every
topological \pinp structure on arises from a smooth \pinp structure, so we can reuse our work from
\cref{almost_spin_classes} to conclude the $\Aut(\xi)$-orbit of $\RP^4$ is again $\pm[\RP^4]$. Therefore the set of
stable diffeomorphism classes is $((\Z/8)/(x\sim -x))\times\Z/2$, which has ten elements, and the triangulation
obstruction and $S'$ are together a complete invariant.
\end{proof}
\subsection{The totally non-spin case}
By \cref{topological_examples}, there is only one normal $1$-type to worry about.
\begin{thm}
\label{top_no_spin}
There are eight stable homeomorphism classes of unorientable, totally non-spin topological $4$-manifolds with
$\pi_1(M)\cong G$. The triangulation obstruction and the Stiefel-Whitney numbers $w_4$ and $w_2^2$ are together a
complete stable homeomorphism invariant.
\end{thm}
Again, this can be extracted from a theorem of Davis~\cite[Theorem 2.3]{Dav05}, who uses a different but equivalent
set of invariants.
\begin{proof}
Following the same line of reasoning as in \cref{non_spin_classes}, \cref{topological_examples} tells us we only
have one normal $1$-type, and its Thom spectrum is $\MSTop\wedge (B\Z/2)^{\sigma-1}$.
\begin{lem}
There is an equivalence $\MTop\simeq\MSTop\wedge (B\Z/2)^{\sigma-1}$.
\end{lem}
\begin{proof}
The proof goes through as in the smooth case, since we have a determinant map and the fact that for any
$\Top$-bundle $P\to M$, $P\oplus\Det(P)$ is canonically oriented, analogously to the smooth case.
\end{proof}
So we need to calculate $\Omega_4^\Top$.
\begin{prop}
\label{top_4_bordism}
$\Omega_4^\Top\cong (\Z/2)^{\oplus 3}$, with a basis given by the classes of $\RP^4$, $\RP^2\times\RP^2$, and
$E_8$. The Stiefel-Whitney numbers $w_4$ and $w_2^2$ and the triangulation obstruction are linearly independent on
this bordism group.
\end{prop}
\begin{proof}
Draw the Atiyah-Hirzebruch spectral sequence computing $\Omega_4^{\Top}$ as $\Omega_4^{\STop}((B\Z/2)^{\sigma-1})$.
It collapses for degree reasons in total degree $4$ and below, and the $4$-line of the $E_\infty$-page has order
$8$. Therefore it suffices to find three linearly independent nonzero elements of $\Omega_4^{\Top}$, which can be
done by computing $w_4$, $w_2^2$, and the triangulation obstruction on $\RP^4$, $\RP^2\times\RP^2$, and $E_8$.
\end{proof}
Just as in the smooth case, $\Aut(\xi)$ acts trivially.
\end{proof}
\Cref{trichotomy} also applies in the topological case: the three normal $1$-types for unorientable topological
manifolds with $\pi_1(M)\cong G$ are precisely the cases where $M$ has a topological \pinp structure,
$M$ has a topological \pinm structure, and $M$ has neither.
\newcommand{\etalchar}[1]{$^{#1}$}


\begin{thebibliography}{KPMT20}

\bibitem[ABP69]{ABP69}
D.~W. Anderson, E.~H. Brown, Jr., and F.~P. Peterson.
\newblock Pin cobordism and related topics.
\newblock {\em Comment. Math. Helv.}, 44:462--468, 1969.

\bibitem[BG97]{BG97}
Boris Botvinnik and Peter Gilkey.
\newblock The {G}romov-{L}awson-{R}osenberg conjecture: the twisted case.
\newblock {\em Houston J. Math.}, 23(1):143--160, 1997.

\bibitem[CS71]{CS71}
Sylvain~E. Cappell and Julius~L. Shaneson.
\newblock On four dimensional surgery and applications.
\newblock {\em Comment. Math. Helv.}, 46:500--528, 1971.

\bibitem[CS76]{CS76}
Sylvain~E. Cappell and Julius~L. Shaneson.
\newblock Some new four-manifolds.
\newblock {\em Ann. of Math. (2)}, 104(1):61--72, 1976.

\bibitem[CS11]{CS11}
Diarmuid Crowley and Jörg Sixt.
\newblock Stably diffeomorphic manifolds and {$l_{2q+1}(\Z[\pi])$}.
\newblock {\em Forum Math.}, 23(3):483--538, 2011.
\newblock \url{https://arxiv.org/abs/0808.2008}.

\bibitem[Dav05]{Dav05}
James~F. Davis.
\newblock The {B}orel/{N}ovikov conjectures and stable diffeomorphisms of
  4-manifolds.
\newblock In {\em Geometry and topology of manifolds}, volume~47 of {\em Fields
  Inst. Commun.}, pages 63--76. Amer. Math. Soc., Providence, RI, 2005.

\bibitem[Deb21]{Deb21}
Arun Debray.
\newblock Invertible phases for mixed spatial symmetries and the fermionic
  crystalline equivalence principle.
\newblock 2021.
\newblock \url{https://arxiv.org/abs/2102.02941}.

\bibitem[FNOP19]{FNOP19}
Stefan Friedl, Matthias Nagel, Patrick Orson, and Mark Powell.
\newblock A survey of the foundations of four-manifold theory in the
  topological category.
\newblock 2019.
\newblock \url{https://arxiv.org/abs/1910.07372}.

\bibitem[Fre82]{Fre82}
Michael~Hartley Freedman.
\newblock The topology of four-dimensional manifolds.
\newblock {\em J. Differential Geometry}, 17(3):357--453, 1982.

\bibitem[Gia73]{Gia73a}
V.~Giambalvo.
\newblock Pin and {$\mathrm{Pin}'$} cobordism.
\newblock {\em Proc. Amer. Math. Soc.}, 39:395--401, 1973.

\bibitem[GOP{\etalchar{+}}20]{GOPWW18}
Meng Guo, Kantaro Ohmori, Pavel Putrov, Zheyan Wan, and Juven Wang.
\newblock Fermionic finite-group gauge theories and interacting
  symmetric/crystalline orders via cobordisms.
\newblock {\em Comm. Math. Phys.}, 376(2):1073--1154, 2020.
\newblock \url{https://arxiv.org/abs/1812.11959}.

\bibitem[Gra80]{Gra80}
Brayton Gray.
\newblock Products in the {A}tiyah-{H}irzebruch spectral sequence and the
  calculation of {$M{\rm SO}\sb\ast $}.
\newblock {\em Trans. Amer. Math. Soc.}, 260(2):475--483, 1980.

\bibitem[HH19]{HH19}
Ian Hambleton and Alyson Hildum.
\newblock Topological 4-manifolds with right-angled {A}rtin fundamental groups.
\newblock {\em J. Topol. Anal.}, 11(4):777--821, 2019.
\newblock \url{https://arxiv.org/abs/1411.5662}.

\bibitem[HKT09]{HKT09}
Ian Hambleton, Matthias Kreck, and Peter Teichner.
\newblock Topological 4-manifolds with geometrically two-dimensional
  fundamental groups.
\newblock {\em J. Topol. Anal.}, 1(2):123--151, 2009.
\newblock \url{https://arxiv.org/abs/0802.0995}.

\bibitem[KLPT17]{KLPT17}
Daniel Kasprowski, Markus Land, Mark Powell, and Peter Teichner.
\newblock Stable classification of 4-manifolds with 3-manifold fundamental
  groups.
\newblock {\em J. Topol.}, 10(3):827--881, 2017.
\newblock \url{https://arxiv.org/abs/1511.01172}.

\bibitem[KPMT20]{KPMTD20}
Justin Kaidi, Julio Parra-Martinez, and Yuji Tachikawa.
\newblock {Topological superconductors on superstring worldsheets}.
\newblock {\em SciPost Phys.}, 9:10, 2020.
\newblock With a mathematical appendix by Arun Debray.
  \url{https://www.scipost.org/SciPostPhys.9.1.010}.

\bibitem[KPT20]{KPT20}
Daniel Kasprowski, Mark Powell, and Peter Teichner.
\newblock Algebraic criteria for stable diffeomorphism of spin 4-manifolds.
\newblock 2020.
\newblock \url{https://arxiv.org/abs/2006.06127}.

\bibitem[Kre84]{Kre84}
M.~Kreck.
\newblock Some closed {$4$}-manifolds with exotic differentiable structure.
\newblock In {\em Algebraic topology, {A}arhus 1982 ({A}arhus, 1982)}, volume
  1051 of {\em Lecture Notes in Math.}, pages 246--262. Springer, Berlin, 1984.

\bibitem[Kre99]{Kre99}
Matthias Kreck.
\newblock Surgery and duality.
\newblock {\em Ann. of Math. (2)}, 149(3):707--754, 1999.
\newblock \url{https://arxiv.org/abs/math/9905211}.

\bibitem[KT90a]{KT90Pinp}
R.~C. Kirby and L.~R. Taylor.
\newblock A calculation of {$\Pin^+$} bordism groups.
\newblock {\em Commentarii Mathematici Helvetici}, 65(1):434--447, Dec 1990.

\bibitem[KT90b]{KT90}
R.~C. Kirby and L.~R. Taylor.
\newblock {\em $\mathit{Pin}$ structures on low-dimensional manifolds}, volume
  151 of {\em London Math. Soc. Lecture Note Ser.}, pages 177--242.
\newblock Cambridge Univ. Press, 1990.

\bibitem[Kur01]{Kur01}
Ichiji Kurazono.
\newblock Cobordism group with local coefficients and its application to
  4-manifolds.
\newblock {\em Hiroshima Math. J.}, 31(2):263--289, 2001.

\bibitem[Ped17]{Ped17}
Riccardo Pedrotti.
\newblock Stable classification of certain families of four-manifolds.
\newblock Master's thesis, Max Planck Institute for Mathematics, 2017.

\bibitem[Pet68]{Pet68}
F.~P. Peterson.
\newblock {\em Lectures on Cobordism Theory}.
\newblock Lectures in Mathematics. Kinokuniya Book Store Co., Ltd., 1968.

\bibitem[Pol13]{Pol13}
Wojciech Politarczyk.
\newblock 4-manifolds, surgery on loops and geometric realization of {T}ietze
  transformations.
\newblock 2013.
\newblock \url{https://arxiv.org/abs/1303.6502}.

\bibitem[Reu20]{Reu20}
David Reutter.
\newblock Semisimple 4-dimensional topological field theories cannot detect
  exotic smooth structure.
\newblock 2020.
\newblock \url{https://arxiv.org/abs/2001.02288}.

\bibitem[Sch76]{Sch76}
Martin~G. Scharlemann.
\newblock Transversality theories at dimension four.
\newblock {\em Invent. Math.}, 33(1):1--14, 1976.

\bibitem[Spa03]{Spa03}
Fulvia Spaggiari.
\newblock On the stable classification of {S}pin four-manifolds.
\newblock {\em Osaka J. Math.}, 40(4):835--843, 2003.

\bibitem[Sto88]{Sto88}
Stephan Stolz.
\newblock Exotic structures on {$4$}-manifolds detected by spectral invariants.
\newblock {\em Invent. Math.}, 94(1):147--162, 1988.

\bibitem[Tei92]{Tei92}
Peter Teichner.
\newblock {\em Topological 4-manifolds with finite fundamental group}.
\newblock PhD thesis, University of Mainz, 1992.
\newblock \url{https://math.berkeley.edu/~teichner/Papers/phd.pdf}.

\bibitem[Tei93]{Tei93}
Peter Teichner.
\newblock On the signature of four-manifolds with universal covering spin.
\newblock {\em Math. Ann.}, 295(4):745--759, 1993.

\bibitem[Tho54]{ThomThesis}
René Thom.
\newblock {\em Quelques propriétés globales des variétés differentiables}.
\newblock PhD thesis, University of Paris, 1954.

\bibitem[Wal64]{Wal64}
C.~T.~C. Wall.
\newblock On simply-connected {$4$}-manifolds.
\newblock {\em J. London Math. Soc.}, 39:141--149, 1964.

\bibitem[Wan95]{Wan95}
Zhenghan Wang.
\newblock Classification of closed nonorientable {$4$}-manifolds with infinite
  cyclic fundamental group.
\newblock {\em Math. Res. Lett.}, 2(3):339--344, 1995.

\bibitem[WWZ20]{WWZ20}
Zheyan Wan, Juven Wang, and Yunqin Zheng.
\newblock Higher anomalies, higher symmetries, and cobordisms ii: {L}orentz
  symmetry extension and enriched bosonic/fermionic quantum gauge theory.
\newblock {\em Ann. Math. Sci. Appl.}, 5(2):171--257, 2020.
\newblock \url{https://arxiv.org/abs/1912.13504}.

\end{thebibliography}
\end{document}